\documentclass[10pt]{amsart}
\theoremstyle{plain}
\usepackage{amsmath, amsfonts, amsthm, amssymb,amscd,bbold}
\usepackage{graphics}
\usepackage{color}
\usepackage{epsfig}
\usepackage[all]{xy}
\graphicspath{ {Figures/} }
\title[LR]{A Littlewood-Richardson rule for the MacDonald inner product 
and bimodules over wreath products.}

\author{Erik Carlsson}
\author{Anthony M. Licata}

\theoremstyle{plain}
\newtheorem{theorem}{Theorem}
\newtheorem{lemma}{Lemma}
\newtheorem{proposition}{Proposition}

\theoremstyle{definition}

\newcommand{\Z}{\ensuremath{\mathbb{Z}}}

\newcommand{\C}{\ensuremath{\mathbb{C}}}

\newcommand{\cF}{\ensuremath{\mathcal{F}}}

\newcommand{\Hom}{{\rm Hom}}

\newcommand{\sltwo}{\ensuremath{\mathfrak{sl}_2}}

\usepackage[vcentermath,enableskew]{youngtab}
\usepackage{young}

\newcommand{\lang}{\left\langle}
\newcommand{\rang}{\right\rangle}

\newcommand{\zee}{\mathfrak{z}}

\newcommand{\ka}{\kappa}
\newcommand{\la}{\lambda}
\newcommand{\kyb}{\small \begin{Young}}
\newcommand{\kye}{\end{Young} \normalsize}

\DeclareMathOperator{\End}{End}

\DeclareMathOperator{\aut}{aut}

\DeclareMathOperator{\Tab}{Tab}

\newtheorem*{thma}{Theorem A}
\newtheorem*{thmb}{Theorem B}

\newtheorem{ex}{Example}

\definecolor{chalkboard}{rgb}{0,.1,.1}
\begin{document}

\bibliographystyle{plain}
\maketitle

\begin{abstract}
We prove a Littlewood-Richardson type formula for
$(s_{\la/\mu},s_{\nu/\ka})_{t^k,t}$, the pairing 
of two skew Schur functions in the
MacDonald inner product at $q=t^k$ for positive integers $k$. 
This pairing counts graded decomposition numbers in 
the representation theory of wreath products of the algebra $\C[x]/x^k$ and symmetric groups.
%We show that
%the coefficients encode a decomposition of the convolution product
%of indecomposable bimodules over the wreath product of the algebra $\C[x]/x^k$ and the symmetric group. 
%This has applications to categorification and certain Nakajima quiver varieties, which we leave for future papers.
\end{abstract}

\section{Introduction}

Let $\Lambda$ denote the algebra of symmetric functions, endowed with the standard bilinear form with respect to which the Schur basis $\{s_\lambda\}$ is orthonormal.  The Littlewood-Richardson rule gives an enumerative formula for
the inner products
$$
	c^{\lambda}_{\mu\nu} = \langle s_\nu^*s_\lambda, s_\mu \rangle,
$$
which are known as Littlewood-Richardson coefficients.  Here $s_\nu^*$ is the linear  operator on symmetric functions adjoint to multiplication by the Schur function $s_\nu$.  (We refer to \cite{F,Mac} for detailed treatments of the Littlewood-Richardson rule).  
%The integrality of the $c^\lambda_{\mu\nu}$ admits several different conceptual explanations. 
The $c^{\la}_{\mu\nu}$ are non-negative integers, and they enumerate tableaux satisfying certain conditions.
The integrality of the Littlewood-Richardson coefficients is also manifest in their appearance as tensor product multiplicities in the representation theory of symmetric groups and as intersection numbers in the Schubert calculus of Grassmannians.
% the enumerative interpretation of the Littlewood-Richardson coefficients is as a count of certain kinds of Tableaux.  
A mild generalization of the Littlewood-Richardson rule adds a fourth partition to the picture: consider the 
algebra $\mathcal{H}_\Lambda$ of operators on symmetric functions spanned by the operators $\{s_\mu s_{\kappa}^*\}_{\mu,\kappa}$.
Define $c^{\kappa\lambda}_{\mu\nu}$ to be structure constants in the expansion
\begin{equation}
\label{zel}
s_\nu^*s_\lambda = \sum_{\mu,\kappa} c^{\kappa\lambda}_{\mu\nu} 
s_\mu s_{\kappa}^*.
\end{equation}
For $\kappa = \emptyset$ these are the ordinary 
Littlewood-Richardson coefficients; an enumerative formula for the general case
was found by Zelevinsky \cite{Z} in the language of \emph{pictures}.
The algebra $H_\Lambda$ is a Hopf algebra; in fact it is the Heisenberg double of the Hopf algebra $\Lambda$.  Thus the Littlewood-Richardson coefficients may also be thought of as structure constants in the canonical basis of the Hopf algebra $H_\Lambda$.

%in general they are still 
%positive integers given  by counting tableaux of certain kind
%(NEED: is there an external reference for this generalization?).

Let $\Lambda_{q,t}$ denote the algebra of symmetric polynomials
over the two-variable coefficient ring $\C(q,t)$,
together with the MacDonald inner product $(\,,)_{q,t}$,
%denote the MacDonald inner product
%on the ring $\Lambda_{q,t}$ of 
%symmetric functions, defined over the two-variable coefficient ring 
which specializes to the standard inner product at $q=t$. 
We will be interested in the specialization $\Lambda_{t^k,t}$
for some integer $k\geq 1$.  The ring $\Lambda_{t^k,t}$ appears in several other mathematical
contexts, including the representation theory of quantum affine algebras.  In particular, several important bases of 
$\Lambda_{t^k,t}$, such as the Schur basis, 
should be related to important bases in the representation theory 
of quantum affine algebras and in the geometry of quiver varieties of affine type.  As a result, it is natural 
to suspect that much of the positive integral structure appearing the 
ordinary theory of symmetric functions will admit an interesting 
generalization from $\Lambda$ to $\Lambda_{t^k,t}$.  The first goal of the present paper is to extend \eqref{zel} to the ring $\Lambda_{t^k,t}$,
in which the dual is now taken with respect to $(\,,)_{t^k,t}$ instead of
the standard inner product.  As the inner product on  $\Lambda_{t^k,t}$ takes values in the ring $\C(t)$,
the precise statement involves a $t$-weighted count of tableaux.  In Section \ref{combsection}, we define $k$-tableaux, which are fillings of a Young diagram with entries which are monomials of the form $at^m$, $0\leq m \leq k$.  A $k$-tableau $T$ has an associated statistic $c(T)$, the degree of the product of all its entries.
Our main theorem then states:
\begin{thma}\label{thm:A}
We have 
\[s^*_\nu s_\la = \sum_{\mu,\ka} c^{\ka\la}_{\mu\nu}(t) s_\mu s_\ka^*,\quad
c^{\ka\la}_{\mu\nu}(t) = \sum_{T} c(T).\]
\end{thma}
\noindent
Here $c(T)$ is the $t$-degree of the tableau $T$, and $T$ ranges over $k$-tableaux.
%defined in Section \ref{combsection}.
When $k=1$ and $\ka$ is the empty partition, the above statement reduces to the usual 
Littlewood-Richardson rule. The proof of this theorem is given in Section 
\ref{combsection}, which deals only with combinatorics. We also show in 
proposition \ref{prop:unimod} that these coefficients are symmetric and 
unimodal, something that is not obvious from enumerative description of 
theorem A.

In Section \ref{catsection} we identify an integral form $\Lambda_{t^k,t}^\Z$ of
$\Lambda_{t^k,t}$ with the Grothendieck
group of graded projective modules over an $S_n$-equivariant
graded ring. 
%The ring, together with this inner product, will be denoted by
%by $\Lambda_{t^k,t}$. 
Under this identification,
the Schur polynomial $s_\lambda$ corresponds to certain a indecomposable projective module $S_\lambda$,  
and $(s_\mu,s_\nu)_{t^k,t}$ measures the graded dimensions of
the $\Hom(S_\mu,S_\lambda)$ up to a grading shift. In particular, this implies
that this graded dimension is a Laurent polynomial in $t$ with nonnegative 
integer coefficients, explaning the positive-integral structure of the 
generalised Littlewood-Richardson coefficients $c^{\ka\la}_{\mu\nu}(t)$.
Specifically, let $A_k = \C[x]/x^k$, and let
$A_k^{[n]}$ denote the smash product of $A_k$ 
with $\C[S_n]$.  The algebras $A_k^{[n]}$ are graded, and we 
consider the Grothendieck group $K(A_k^{[n]}-\mbox{gmod})$ of 
finitely generated projective $A_k^{[n]}$ modules.  This space 
is a free $\Z[t,t^{-1}]$ module, where multiplication by $t$ 
corresponds to a shift in the grading.  Since the $\Hom$ pairing in the category of graded $A_k^{[n]}$ modules induces a semi-linear pairing on 
the Grothendieck group, we slightly modify the bilinear form on
$\Lambda_{t^k,t}^\Z$ to be semi-linear in $t$.  Our second main theorem is then the following.

\begin{thmb}
There is an isometric isomorphism
\[\Phi:  \bigoplus_{n=0}^\infty K(A_k^{[n]}-\mbox{gmod}) 
\longrightarrow \Lambda_{t^k,t}^\Z,\]
where the bilinear form on the left hand side is induced 
from the $\Hom$ bifunctor.
\end{thmb}
\noindent 
As a result, we obtain an interpretation of the polynomials $c^{\ka\la}_{\mu\nu}(t)$ as decomposition
numbers in the convolution product of explicit bimodules over the rings $A^{[n]}_k$. 
When $k=1$ there is no interesting grading on the module category, and the identification above reduces to the well-known isomorphism
between $\Lambda$ and the Grothendieck group of representations of all symmetric groups \cite{G}.  
%The algebra $sym$ admits a well-known interpretation via 
%the representation theory of symmetric groups.  
%Explicitly, there is an isometric isomorphism
%
%\[\Phi : \bigoplus_{n=0}^\infty K(\C[S_n]-\mbox{mod}) \longrightarrow \Lambda\]
%
%where $ K(\C[S_n]-\mbox{mod})$ is the Grothendieck group of the 
%category of finitely-generated modules over the group algebra of 
%the symmetric group $S_n$.

The graded rings $A_k^{[n]}$ appear in several other representation theoretic contexts.  For example, these algebras for $k=2$ play a central role in the categorification  \cite{CL1}  of the Heisenberg double of $\Lambda_{t^2,t}$ and in the level one quantum affine categorifications of \cite{CL2}; the algebras $A_k^{[n]}$ for $k > 2$ should appear in the higher level analogs of those constructions.  In fact, part of our original motivation for considering the polynomials $c^{\ka\la}_{\mu\nu}(t)$ was to combinatorially compute the structure constants for multiplication in the canonical basis of the quantum Heisenberg algebra considered in \cite{CL1}; these structure constants are given by the $k=2$ case of Theorem A.  Similarly, the generalized Littlewood-Richardson coefficients $c^{\ka\la}_{\mu\nu}(t) $ should all appear as structure constants for multiplication in the canonical basis of certain infinite dimensional Hopf algebras.  Another closely related appearance of $A_k^{[n]}$ involves category $\mathcal{O}$ for the rational Cherednik algebra of the complex reflection group $(\Z/k\Z) \wr S_n$ at integral parameters.  The algebra $A_k^{[n]}$ is an example of an Ariki-Koike algebra, and the study of $\mathcal{O}$ via the KZ functor involves mapping $\mathcal{O}$ to a category of $A_k^{[n]}$ modules.  From this point of view, Theorems A and B together have applications to combinatorial description of hom spaces between projective modules in $\mathcal{O}$, though we have not fully explored this here.  Relationships between higher-level Heisenberg categorification and category $\mathcal{O}$ for rational Cherednik algebras appear in \cite{SV}.

The bilinear form on $\Lambda_{t^k,t}$, the ring $A_k^{[n]}$, and the generalized Littlewood-Richardson coefficients $c^{\ka\la}_{\mu\nu}(t) $ all depend on the choice of positive integer $k$.  It might be interesting to formulate versions of Theorems A and B in a way that does not require $k$ to be a positive integer and thus recover the variable $q$ in the Macdonald theory.   Notice that a 
rational function in $\C(q,t)$ is determined by its values at $q=t^k$ at all positive integers $k$, so in a sense we lose no information
by specializing.

\subsection{Acknowledgements}
The authors would like to thank Sabin Cautis, Josh Sussan, Ben Webster, Stephen Morgan, and Alistair Savage for a number of helpful conversations.

\section{The Littlewood-Richardson rule}
\label{combsection}
\subsection{Notations}
For all notations in this section, we have followed
MacDonald's book \cite{Mac}. 

Given partitions $\ka,\la,\mu,\nu$, let $\Tab(\la-\mu,\nu)$ denote
the set of semi-standard Young tableaux of shape $\la-\mu$, and 
content $\nu$. The \emph{word} $w(T)$ of a tableau $T$ is the set of
numbers in the diagram read from right to left, top to bottom.
For instance, the word of the tableau

\[\kyb
&&1&2&4&4&4 \cr
&&3&4 \cr
1&2&2 \cr
\kye\]
% \young(::12444,::34,122)
%

\noindent
is $4442143221$. The word of a partition $w(\la)$
is defined as the word of the tableau of shape $\la$ 
in which row $i$ is filled with the number $i$. 
Let $\Tab^0(\la-\mu,\nu)$ denote the subset of
tableau whose word $a_1\cdots a_n$ 
is a \emph{lattice permutation}, meaning 
that for each $i,k$, the number of occurrences of $i$ in 
$a_1\cdots a_k$ is greater than or equal to the number of occurrences
of $i+1$. More generally, we define $\Tab^0(\la-\mu,\nu-\ka)$ to
be the set of tableaux $T$ with content $\nu-\ka$
such that the concatenated word
$w(\ka)w(T)$ is a lattice
permutation.

Fix an integer $k\geq 1$, and define a $k$-tableau of shape $\la-\mu$ 
to be a labeling of the boxes of $\la-\mu$ with monomials of
of the form $a t^b$ for $a\geq 1$, $0 \leq b\leq k-1$. 
We define a total order on these monomials by
\begin{equation}
\label{leqt}
a t^b \leq c t^d \Longleftrightarrow b<d \mbox{ or } b=d \mbox{ and } a \leq c,
\end{equation}
which is the same as the ordering obtained by replacing $t$ by
a large positive number. Call $T'$ the tableau obtained from $T$
by setting $t=1$,
%We define the word of such a tableau by $w(T)=w(T')$, where $T'$
%is the tableau obtained from $T$ by setting $t=1$. 
and let $\Tab_k(\la-\mu,\nu-\ka)$ be the set of $k$-tableaux
which are semistandard with respect to \eqref{leqt}, such that
content of $T'$ is $\nu-\ka$. 
We also define a statistic on $k$-Tableau by
\begin{equation}
c(T) = \prod_{a t^b \in T} t^b.
\end{equation}

Any $k$-tableau $T$ corresponds to 
a sequence of regular tableaux $T^{i}$ for $i\geq 0$, defined as 
the subdiagram of coefficients in 
all boxes containing $at^i$ for some $a$. For instance, if
\[T=\kyb
&&&$3$&$2t$&$2t$ \cr
&&$1$& $t$ & $t^2$ \cr
$t^2$& $t^2$\cr
\kye,\]
then
\[T^0=\kyb
&&&$3$&& \cr
&&$1$ &&\cr
& \cr
\kye,\quad 
T^1=
\kyb
&&&&$2$&$2$ \cr
&&& $1$ &  \cr
&\cr
\kye,\quad
T^2=\kyb
&&&&& \cr
&&&& $1$ \cr
$1$& $1$\cr
\kye.\]
%
%Define the word of a $k$-tableau by
%
%\[w(T) = w(T^0)w(T^1)\cdots\]
%

Let $\Tab_k^0(\la-\mu,\nu-\ka)$ denote the subset of $k$-tableaux such that 
\[w(\ka)w(T^0)w(T^1)\cdots\] 
is a lattice permutation. For instance,
\[\kyb
&&$2$&$3$&$2t^2$\cr
&$t$&$t$\cr
$1$ & $2t^2$\cr
\kye \in \Tab^0_3([532]-[21],[442]-[21])\]
because $1123211122$ is a lattice permutation,
and the semistandardness condition is satisfied. 
The statistic is $c(T)=t^6$.

\subsection{The main theorem}

Consider the space of symmetric polynomials $\Lambda$ in
infinitely many variables, and let $p_\mu$, $e_\mu$, $h_\mu$, and $s_\mu$
denote the power sum, elementary, complete, and Schur 
bases respectively.  We denote by $\Lambda^\Z = \mbox{span}_{\Z[t,t^{-1},q,q^{-1}]}\{s_\mu\}$ the integral form of $\Lambda$ spanned by the Schur functions.
The MacDonald inner product on $\Lambda$ is defined in the 
power sum basis by
\[(p_\mu,p_\nu)_{q,t} = \delta_{\mu\nu} \zee(\mu) 
\prod_{j} \frac{1-q^{\mu_j}}{1-t^{\mu_j}},\]
where
\[p_\mu = \prod_j p_j,\quad p_j = \sum_i x_i^j,\quad
\zee(\mu) = \aut(\mu)\prod_j \mu_j.\]
For a fixed integer $k$,
and any symmetric polynomial $f \in \Lambda$, define
a dual multiplication operator by
\[(f^* g,h)_{t^k,t} = (g,fh)_{t^k,t}.\]

We may now state the main theorem:
\begin{theorem}
\label{tabthm}
Fix a positive integer $k$. There exist unique coefficients
$c^{\ka\la}_{\mu\nu}(t)$ satisfying
\[s^*_\nu s_\la = \sum_{\mu,\ka} c^{\ka\la}_{\mu\nu}(t) s_\mu s_\ka^*.\]
Furthermore, 
\begin{equation}
\label{CR}
c^{\ka\la}_{\mu\nu}(t) = \sum_{T \in \Tab^0_k(\la-\mu,\nu-\ka)} c(T),
\end{equation}
if $\mu \subset \la$ and $\ka \subset \nu$, or zero otherwise.
\end{theorem}

\begin{ex} Take $k=2,\ka=[1],\la=[32],\mu=[1],\nu=[32]$.
By lemma \ref{skewlemma} below we have
\[c^{\ka \la}_{\mu\nu}(t) = (s_{\la/\mu},s_{\nu/\ka})_{t^2,t}=2+5t+7t^2+5t^3+2t^4.\]
Which equals 21 at $t=1$.
On the other hand, there are 25 elements of $\Tab_2([32]-[1],[32]-[1])$.
The remaining four that do not satisfy the lattice word condition are
\[
\kyb
&$2$&$2$ \cr
$1$ & $t$ \cr
\kye,\quad
\kyb
&$2$&$2$\cr
$t$ & $t$ \cr
\kye,\quad
\kyb
&$2$&$t$ \cr
$2$&$t$ \cr
\kye, \quad
\kyb
&$2$&$2t$ \cr
$t$ & $t$ \cr
\kye.
\]
\end{ex}
\begin{ex} \label{strip}
If $\nu=(n),\la=(m)$ consist of a horizontal
strip, then $c^{\ka(m)}_{\mu(n)}(t)$ is zero unless 
$\ka=(n-l),\mu=(m-l)$ for some $l\geq 0$. In this case
\[\Tab^0_k(\la-\mu,\nu-\ka)=\Tab_k(\la-\mu,\nu-\ka)\]
because the coefficients of every boxare one, so that the lattice word
condition is always satisfied. By stars and bars, we have
\[c^{(n-l)(m)}_{(m-l),(n)}(t) = %{k+l-1 \choose l}_{t}=
\frac{(1-t^{k+l-1})\cdots(1-t^k)}{(1-t)\cdots(1-t^l)}.\]
\end{ex}

Before proving the theorem, we need a couple of lemmas.
Let %$s_\la$ denote the Schur polynomial, and 
$s_{\la/\mu}$
denote the skew Schur polynomial. More generally, if $u_\mu$ is
any basis of $\Lambda$, we define
\[u_{\la/\mu} = u^*_\mu u_\la\]
where the dual is taken with respect to the \emph{standard} inner product,
$k=1$.

\begin{lemma}
\label{skewlemma}
We have
\[c^{\ka\la}_{\mu\nu}(t) = (s_{\la/\mu},s_{\nu/\ka})_{t^k,t}.\]
\end{lemma}

\begin{proof}
It is straightforward to check this relation when the Schur basis
is replaced by the power sums on the right hand side, and in the
definition of $c_{\mu\nu}^{\ka\la}(t)$. 
Then simply apply the change of basis
matrix from $p$ to $s$ to each of the coordinates in the tensor
$c^{\ka\la}_{\mu\nu}(t)$.

\end{proof}

\begin{lemma}
\label{nl}
We have
\[(s_{\la/\mu},h_{\nu})_{t^k,t} = \sum_{T \in \Tab_k(\la-\mu,\nu)} c(T).\]
\end{lemma}

\begin{proof}
We begin by rewriting
\begin{equation}
\label{h2g}
(s_{\la/\mu},h_\nu)_{t^k,t} = (h_\nu^*s_{\la/\mu},1)=
(s_{\la},g_{\nu} s_{\mu}),
\end{equation}
where $g_\mu$ is %the dual basis to the monomial symmetric functions
%with respect to the MacDonald inner product $(\ ,)_{t,t^k}$,
%not the one above. This
the image of $h_\mu$ under the homomorphism defined
on generators by
\[\rho_{t^k,t} : p_j \mapsto \frac{1-t^{kj}}{1-t^j} p_j.\]

Now let us proceed by induction on the length of $\nu$.
We start with the base case in which $\nu$ only has one component,
say $\nu_1=m$. We first find the coefficients in the expansion
\begin{equation}
\label{gm}
g_m = \sum_{\pi} a_\pi(t) h_\pi.
\end{equation}
We claim that
%
%\begin{equation}
%\label{amu}
%a_\pi(t) = 
%\frac{1}{\aut(\pi)}\sum_{a_1,...,a_\ell} 
%t^{a_1 \pi_1+\cdots+a_\ell \pi_\ell},
%\quad 0 \leq a_i \leq k-1,
%\end{equation}
%
\begin{equation}
\label{amu}
a_\pi = \sum_{\beta \rightarrow \pi}\ 
\sum_{0 \leq a_1<\cdots<a_\ell\leq k-1} 
t^{a_1 \beta_1+\cdots+a_{\ell} \beta_\ell},
%\quad 0 \leq a_i \leq k-1,
\end{equation}
%
%where $\ell=\ell(\pi)$ is the number of elements of $\pi$,
%and the $a_i$ are all distinct. 
where $\ell=\ell(\pi)$ is the length of $\pi$, and
$\beta$ ranges over all rearrangements of $\pi$, i.e. over the
orbit of $\pi$ under $S_\ell$.
The easiest way to see this
is to notice that
\[g_\mu(x_1,x_2,...) = h_\mu(x_1,x_2,...;tx_1,tx_2,...;t^2x_1,t^2x_2,...),\]
and compare the coefficient of $x^\nu$ with the
right hand side of \eqref{gm}.

Inserting \eqref{gm} and \eqref{amu} into \eqref{h2g}, and using the Pieri rule,
we have
\begin{equation}
\label{basecase}
(s_{\la/\mu},h_m)_{t^k,t} =
\sum_{\beta} \sum_{a_1<\cdots<a_{\ell}} 
t^{a_1\beta_1+\cdots+a_\ell \beta_\ell}|\Tab(\la-\mu,\beta)|.
\end{equation}
In particular, the tableau count is independent of the
ordering of $\beta$.
Finally, there is a bijection
\[\bigcup_{\beta} \left\{a_1<\cdots <a_{\ell}\right\} \times
\Tab(\la-\mu,\beta) \longleftrightarrow \Tab_k(\la-\mu,m),\]
in which a box label $j$ maps to $a_j$.
The statistic $c(T)$ corresponds
to the power of $t$ in \eqref{basecase}, proving the base case.
%The preimage of a point is contained
%in a single component $\pi$, is of size $\aut(\pi)$, and its
%statistic is the monomial under the sum in \eqref{amu}. 
%Combining this with \eqref{h2g}, \eqref{gm}, and 
%the Pieri rule 
%
%\[(s_\la,h_\pi s_\mu) = |\Tab(\la-\mu,\beta)|\]
%
%for any rearrangement $\beta$, establishes the base case.

Now for the induction step, let $\nu=\nu' \cup \nu''$ be any decomposition
of $\nu$, so that $g_\nu=g_{\nu'}g_{\nu''}$. Let us also fix an identification
of the rows of $\nu',\nu''$ with the corresponding row of $\nu$. Since
\[(s_\la,g_\nu s_\mu) = 
\sum_{\pi} (s_\la,g_{\nu'}s_\pi)(s_\pi,g_{\nu''} s_\mu),\]
it suffices to prove that there is a bijection
\begin{equation}
\label{indbij}
\Tab_k(\la-\mu,\nu) \longleftrightarrow
\bigcup_{\pi} \Tab_k(\la-\pi,\nu') \times
\Tab_k(\pi-\mu,\nu''),
\end{equation}
such that $c(T)=c(U)c(V)$ whenever $T$ maps to $(U,V)$.
To do this, partition the left side into groups 
$\Tab_k(\la-\mu,\nu)_\gamma$ indexed by the
multiset $\gamma$ of all monomials $at^b$ in $T$.
Since \eqref{leqt} is a total ordering, any choice of content $\gamma$
induces a multiset $c \subset \Z_{\geq 0}$ such that
\[\big|\Tab_k(\la-\mu,\nu)_\gamma\big| = \big|\Tab(\la-\mu,c)\big|.\]

The choice of $\gamma$ induces
a corresponding decomposition $(\gamma',\gamma'')$
on the right side of \eqref{indbij}, by taking the elements of 
$\gamma',\gamma''$ to be the monomials $at^b \in \gamma$ 
such that the row $\nu_a$ is in $\nu',\nu''$ respectively.
The argument of the preceding paragraph reduces 
the statement to the case $k=1$, which is true by the
usual Littlewood-Richardson rule.
\end{proof}

We may now prove the theorem.

\begin{proof}

The existence and uniqueness statement follow from lemma \ref{skewlemma}.

We first prove the case $\ka=\emptyset$, by extending the
proof of the usual Littlewood-Richardson rule, as it is
explained in \cite{Mac}. Using the expansion of
$s_\nu$ in the monomial basis, which is well known to be the dual basis
to $h_\nu$ under the standard inner product, we have
\[(s_{\la/\mu},h_\nu)_{t^k,t} = 
\sum_{\pi} (s_{\la/\mu},s_\pi)_{t^k,t} \big|\Tab(\pi,\nu)\big|.\]
%
%so by lemma \ref{nl} we have
%
%\[\sum_{T \in \Tab_k(\la-\mu,\nu)} c(T)=
%\sum_{\pi} (s_{\la/\mu},s_\pi)_{t^k,t} \big|\Tab(\pi,\nu)\big|.\]
%
Then by lemma \ref{nl}, and the invertibility of the triangular matrix
$|\Tab(\pi,\nu)|$, it suffices to prove that there is a bijection
\begin{equation}
\label{lrbij}
\Tab_k(\la-\mu,\nu) \longleftrightarrow
\bigcup_{\pi} \Tab^0_k(\la-\mu,\pi) \times \Tab(\pi,\nu),
\end{equation}
producing the correct statistic $c(T)$. 

Each tableau $T \in \Tab_k(\la-\mu,\nu)$ has an associated filtration
\[F : \mu=\mu^1 \subset \cdots \subset \mu^{k+1}=\la,\]
so that $\mu^{i+1}-\mu^{i}$ is the shape of $T^i$.
We obtain a decomposition
\begin{equation}
\label{Fbij}
\Tab_k(\la-\mu,\nu) \longleftrightarrow
\bigcup_{F} \Tab_k(\la-\mu,\nu)_F,
\end{equation}
and similarly for $\Tab^0(\la-\mu,\nu)$, by restriction.
Given such a filtration $F$, let $\la_F$ denote any skew diagram
which is a disconnected union of the shapes $\mu_{i+1}-\mu_i$, positioned in the
plane in some way in order from upper right to lower left. 
For instance, one choice of $\la_F$ might be
\[F=[1]\subset [3,2] \subset [5,2] \subset [5,3,1],\quad
\la_F = \mbox{\young(:::::::\ \ ,::::::\ \ ,::::\ \ ,::\ ,\ )}.\]
For any choice of $\la_F$, we obtain a bijection
\[\Tab_k(\la-\mu,\nu)_F \longleftrightarrow \Tab(\la_F,\nu),\]
that respects the corresponding lattice word conditions.
This reduces \eqref{lrbij} to the case $k=1$, which follows from the usual
algorithm of Littlewood-Robinson-Schensted. 

We now prove the general case. 
%By the lemma \ref{},
%the above argument also estbalishes the theorem whenever $\mu$
%is empty.
Using the usual Littlewood-Richardson rule to expand $s_{\nu/\ka}$, we have
\[(s_{\la/\mu},s_{\nu/\ka})_{t^k,t} = \sum_{\pi}
(s_{\la/\mu},s_{\pi})_{t^k,t} \big|\Tab^0(\nu-\ka,\pi)\big|.\]
Then applying the special case we just proved, it suffices 
to find a suitable bijection
\begin{equation}
\label{crbij}
\Tab^0_k(\la-\mu,\nu-\ka) \longleftrightarrow
\bigcup_{\pi} \Tab^0_k(\la-\mu,\pi) \times 
\Tab^0(\nu-\ka,\pi).
\end{equation}
Once again, the decomposition \eqref{Fbij} reduces this statement 
to the case $k=1$, which was proved by Zelevinsky \cite{Z}.
\end{proof}

%\begin{theorem}\label{combinatorics}
%Define Laurent polynomials $c^{\kappa\lambda}_{\mu\nu}(t)\in \Z[t,t^{-1}]$ by the equation
%$$%
%	s_\nu^*s_\lambda = \sum_{\mu,\kappa} c^{\kappa\lambda}_{\mu\nu}(t) s_\mu s_{\kappa}^*.
%$$
%Then $c^{\kappa\lambda}_{\mu\nu}(t) = \sum_{T \in Tab_k^0(\lambda-\mu,\nu-\kappa)} c(T).$
%\end{theorem}

%\begin{proposition}\label{prop:special case}
%For $\nu = (n)$, $\lambda = (m)$, we have
%$c_{\mu(n)}^{\kappa(m)}(t) = 0$ unless $\kappa = (n-s)$ and $\mu = (m-s)$ are s%ingle part partitions, in which case
%$$%
%	c_{(m-s)(n)}^{(n-s)(m)} = {k+s - 1 \choose s}_t,
%$$	
%\end{proposition}
%\begin{proof}
%We NEED to add this proof. (It should be easy).
%In the above, the quantum binomial coefficient ${x\choose y}_t$ is scaled to be% symmetric about the origin, e.g. ${4\choose 2}_t = t^{-4} + t^{-2} + 2 + t^2 +%t^4$.
%\end{proof}

Now consider the semi-linear form on $\Lambda_{t^k,t}^\Z$ defined by
\begin{equation}
\label{hermsp}
\lang f, g \rang_k = (f,t^{(1-k)d} \overline{g})_{t^{2k},t^2},
\end{equation}
Where the conjugation takes $t\mapsto t^{-1}$.
Its extension to the rationals takes the form
\begin{equation}
\label{hermsp0}
\lang p_\mu,p_\nu\rang_k = \delta_{\mu,\nu}
\zee(\mu)\prod_i \left(t^{(1-k)\mu_i}+t^{(3-k)\mu_i}+\cdots+t^{(k-1)\mu_i}\right),
\end{equation}
establishing that it is Hermitian.
With respect to this inner product,
we then have
\[s_\nu^*s_\lambda = \sum_{\mu,\kappa} C^{\kappa\lambda}_{\mu\nu}(t)
s_\mu s_{\kappa}^*,\]
where
\[C^{\ka\la}_{\mu\nu}(t)=t^{(1-k)(|\la|-|\mu|)}c^{\kappa\lambda}_{\mu\nu}(t^2).\]

Call a Laurent polynomial 
\[f(t) = a_{-l}t^{-l} + a_{2-l} t^{2-l} + \dots + a_{l-2} t^{l-2} + a_l t^l\]
symmetric unimodal if for all $f(t)=f(t^{-1})$, and
$a_i \leq  a_{i+2}$ for $i<0$.

\begin{proposition} \label{prop:unimod}
The polynomials $C^{\ka\la}_{\mu\nu}(t)$ are symmetric unimodal.
\end{proposition}
\begin{proof}

First, the symmetry statement follows from
lemma \ref{skewlemma}, and the obvious symmetry of \eqref{hermsp0}.

Since the sum of unimodal expressions is unimodal, it suffices to check the unimodality of $\lang s_\mu,s_\nu\rang_k$,
by lemma \ref{skewlemma}, and the Schur-positivity of the skew
Schur functions. We have
\[\lang s_\mu,s_\nu\rang_k= (s_\mu,t^{(1-k)d}\rho_{t^{2k},t^2} s_\nu)=
\sum_{\ka,\la} a_{\ka\la}s_\ka(t^{1-k},t^{3-k}...,t^{k-1})
(s_\mu,s_\la),\]
where $a_{\ka\la}$ are the multiplicities of the decomposition
into irreducibles
\[\mathbb{S}_\nu(U \boxtimes V) = \bigoplus_{\ka,\la}
a_{\ka\la} \mathbb{S}_\ka(U)\boxtimes \mathbb{S}_\la(V)\]
over $GL(U)\times GL(V)$. In particular, they are nonnegative integers.
The answer now follows from the unimodality of
$s_\ka(t^{1-k},t^{3-k}...,t^{k-1})$, 
see \cite{Mac} chapter I, section 8, example 4.

\end{proof}
\noindent

A notable feature of Proposition \ref{prop:unimod} is that neither symmetry 
nor unimodality is immediately clear from theorem \ref{tabthm}.
It would be interesting to give a purely enumerative proof
in this way, or a representation-theoretic proof, along the lines
of the next section.

\section{Categorification}
\label{catsection}
\subsection{Wreath products of $H^*(\mathbb{P}^{k-1})$}
Fix the integer $k \geq 1$, and let $A_k = H^*(\mathbb{P}^{k-1},\C) \cong \C[x]/x^k$.  Denote by $A_k^{[n]}$ the smash product
of $A_k$ with the group algebra of the symmetric group $S_n$,
$$
	 A_k^{[n]} = A_k \# \C[S_n].
$$
As a vector space, $A_k^{[n]} = A_k^{\otimes n} \otimes_\C \C[S_n]$, and by convention we take $A_k^{[0]} = \C$.  We give $A_k^{[n]}$ a grading by declaring the degree of $x$ to be 1, and putting $\C[S_n]$ in degree 0.

Let $K(A_k^{[n]})$ denote the Grothendieck group of the category of $\Z$-graded finitely-generated projective left $A_k^{[n]}$ modules.  The $\Z$-grading on $A_k^{[n]}$ endows $K(A_k^{[n]})$ with the structure of a free $\Z[t,t^{-1}]$ module, where shifting by $1$ in the internal grading of a module corresponds to multiplication by $t$ on the class in the Grothendieck group,
$$
	[M\{\pm 1\}] = t ^{\pm1}[M] \in K(A_k^{[n]}).
$$

If we choose a complete set of minimal idempotents $e_\lambda \in \C[S_n]$, so that $\{\C[S_n]e_\lambda\}_{\lambda \vdash n}$ are representatives of the isomorphism classses of irreducible $\C[S_n]$ modules, then $\{A_k^{[n]}e_\lambda\}$ are representatives of the isomorphism classes of indecomposable projective $A_k^{[n]}$ modules, up to grading shift.  It follows that the rank of $K(A_k^{[n]})$ as a free $\Z[t,t^{-1}]$ module is the number of partitions of $n$.  Thus the free $\Z[t,t^{-1}]$ module $K(A_k^{[n]})$
comes equipped with a canonical basis, namely, the classes $\{[A_k^{[n]}e_\lambda]\}_{\lambda \vdash n}$ of the indecomposable projective modules.

Let $\cF_k = \oplus_{n=0}^\infty K(A^{[n]})$.  Homomorphisms in the module category endow  $\cF$ with a bilinear form,
$$
	\langle [X],[Y] \rangle_{\cF_k} = gdim\ \Hom (X,Y)\in \Z[t,t^{-1}],
$$
where $gdim$ $V$ denotes the graded dimension of a $\Z$-graded vector space $V$.  By convention, the individual summands $K(A^{[n]})$ for different $n$ are orthogonal with respect to this bilinear form.  This form is \emph{semi-linear} with respect to $t$,
$$
	\langle t^{\pm 1}[M],[N] \rangle_{\cF_k} = t^{\pm 1} \langle [M],[N] \rangle_{\cF_k} = 
	 \langle [M],t^{\mp 1}[N] \rangle_{\cF_k}.
$$

The following Lemma can be checked easily using the fact that $A_k^{[n]}$ has a nondegerate trace $tr: A_k^{[n]} \longrightarrow \C$.

\begin{lemma}
With respect to the basis $\{[A_k^{[n]}e_\lambda]\}$, the matrix of the bilinear form $\langle \cdot,\cdot \rangle_{\cF_k}$ is symmetric.
\end{lemma}
(The above symmetry together with the semi-linearity together can be thought of as defining a Hermitian inner product on $\cF_k$ after specialising $t$ to a point on the unit circle.)

\subsection{Bimodules}
The standard embeddings $S_m\subset S_n$, $m\leq n$ of small symmetic groups into larger ones give rise to embeddings of algebras
$A_k^{[m]} \subset A_k^{[n]}$.  Let $\lambda \vdash m$ a partition, and let $e_\lambda \in \C[S_m]$ be an associated minimal idempotent in the group algebra.  We view $e_\lambda$ as an element of $A_k^{[n]}$ for any $n>m$ via the embeddings $\C[S_m]\subset A_k^{[m]}\subset A_k^{[n]}$, and set 
$$
	P^\lambda = A_k^{[n]}e_\lambda \text{ and } Q^\lambda = e_\lambda A_k^{[n]}\{ m \}.
$$
The internal grading shift $\{ m \}$ in the definition of $Q^\lambda$ is for convenience; it ensures that various hom spaces occurring later have a grading which is symmetric about the origin.  The space $P^\lambda$ is naturally an $(A_k^{[n]},A_k^{[n-m]})$ bimodule, while the space $Q^\lambda$ is naturally a $(A_k^{[n-m]},A_k^{[n]})$ bimodule.  By convention, we set $P^\lambda = Q^\lambda = 0$ when $m>n$, and $P^\emptyset = Q^\emptyset = A_k^{[n]}$ as an $(A_k^{[n]},A_k^{[n]})$ bimodule.  

We denote tensor products of these bimodules by concatenation, where the tensor product is understood to be over the algebra $A_k^{[l]}$ acting on both sides of the tensor product.  So, for example, if $\lambda \vdash m$ and $\mu \vdash l$, then for each $n\geq \mbox{max}\{m,l\}$,
$$
	P^\lambda Q^\mu  = P^\lambda \otimes_{A_{k^{[n-m]}}} Q^\mu
$$
is an $(A_k^{[n]}, A_k^{[n-m+l]})$ bimodule.  On the other hand,
$$
	Q^\mu P^\lambda = Q^\mu \otimes_{A_k^{[n]}} P^\lambda
$$
is an $(A_k^{[n-l]},A_k^{[n-m]})$ bimodule.  

Denote by $1_0$ the trivial module over $A_k^{[0]}= \C$.  The $\Z[t,t^{-1}]$-module $\cF_k$ comes equipped with various natural bases indexed by partitions $\lambda = (\lambda_1,\lambda_2,\dots,\lambda_r) \vdash n$.  Examples include
\begin{enumerate}
\item $S^\lambda : = P^\lambda 1_0$ (This module was denoted $A_k^{[n]}e_\lambda$ in the previous subsection),
\item $E^\lambda := P^{(1^{\lambda_1})}P^{(1^{\lambda_2})}\dots P^{(1^{\lambda_r})} 1_0$, and
\item $H^\lambda := P^{(\lambda_1)} P^{(\lambda_2)} \dots P^{(\lambda_r)}1_0$.
\end{enumerate}

\subsection{The character map}
For the remainder of this paper, we consider the integral form $\Lambda_{t^k,t}^\Z$
of $\Lambda_{t^k,t}$; by definition $\Lambda_{t^k,t}^\Z$ is the free $\Z[t,t^{-1}]$-module spanned by the Schur functions.
We define a map of $\Z[t,t^{-1}]$ modules
$$
\Phi : \cF_k \longrightarrow \Lambda_{t^k,t}^\Z,  \ \ \Phi([S_\lambda] ) = s_\lambda,
$$	
by sending each canonical basis vector to the corresponding Schur function.  It is straightforward to check that
$\Phi([E^\lambda])$ is the elementary symmetric function $e_\lambda$, while $\Phi([H^\lambda])$ is the complete symmetric function $h_\lambda$.

%It is related to the inner product from the last section by
%
%\begin{equation}
%\label{prod1}
%(f,g)'=
%\end{equation}
%

The proof of the following theorem will be given at the end of this section.

\begin{theorem}\label{categorification}
The map $\Phi$ is an isometry with respect to 
$\lang\,, \rang_{\cF_k}$, $\lang\,, \rang_k$.
\end{theorem} 

The above theorem translates to the statement that the module categories for the algebras $A_k^{[n]}$ for all $n$ together categorify Macdonald's ring of symmetric functions at $q=t^k$.  When $k = 1$, the map $\Phi$ is just the Frobenius character map, and the above theorem is of course well known.  Note that $\Lambda_{t,t}$ is the ring of symmetric functions (over $\Z[t,t^{-1}]$) endowed with a  bilinear form with respect to which the Schur functions $\{s_\lambda\}$ are orthonormal basis.  On the other hand, $A_1^{[n]}=\C[S_n]$ is a semi-simple algebra, whence the classes of indecomposable projective (irreducible) modules give an orthonormal basis in the Grothendieck group.

\subsection{Representation-theoretic interpretation of $c^{\kappa\lambda}_{\mu\nu}(t)$}

As a consequence of Theorem \ref{categorification}, the generalized Littlewood-Richardson coefficients $c^{\kappa\lambda}_{\mu\nu}(t)$ inherit an interpretation as the graded dimension of vector spaces arising in the representation theory of the algebras $A_k^{[n]}$.  To explain this interpretation, we recall the following.

\begin{proposition}
The bimodules $P^\lambda Q^\mu$ are indecomposable.  Any bimodule of the form $Q^\alpha P^\beta$ decomposes as a direct sum of (graded shifts of) the indecomposable bimodules $\{P^\lambda Q^\mu\}_{\lambda, \mu}$. 
\end{proposition}
\begin{proof}
As explained in, for example,  \cite[Proposition 6]{CL1}, it is straightforward to check that $\End( P^\lambda Q^\mu)$ is a non-negatively graded algebra whose degree 0 piece is one-dimensional.  Thus $P^\lambda Q^\mu$ is indecomposable.  The fact that all indecomposable bimodules are of the form $P^\lambda Q^\mu$ also follows just as in \cite[Proposition 6]{CL1}.
\end{proof}

For partitions $\kappa,\lambda,\mu,\nu$, we may therefore define a $\Z$ graded vector $C^{\kappa\lambda}_{\mu\nu}$ as the multiplicity space of $P^\mu Q^\kappa$ in the decomposition of $Q^\nu P^\lambda$ into indecomposable bimodules:

$$
	Q^\nu P^\lambda = \bigoplus_{\mu,\kappa} P^\mu Q^\kappa \otimes_\C C^{\kappa\lambda}_{\mu\nu}.
$$

\begin{theorem}\label{homspace}
The graded dimension of $C^{\kappa\lambda}_{\mu\nu}$ is equal to the generalized Littlewood-Richardson coefficient 
$C^{\ka\la}_{\mu\nu}(t)$.
\end{theorem}
\begin{proof}
This follows immediately from Theorems \ref{tabthm} and 
\ref{categorification}, and the normalization \eqref{hermsp}. 
\end{proof}

In light of Proposition \ref{prop:unimod} and Theorem \ref{homspace} above, it is tempting to speculate that the graded vector space $C^{\kappa\lambda}_{\mu\nu}$ can be endowed with a linear action of the Lie algebra $\sltwo$ in a way that aligns the weight space decomposition with the grading; this would give a more conceptual explanation of the symmetry and unimodality of these coefficients.
\
\subsection{Proof of Theorem \ref{categorification}}
We now give the proof of Theorem \ref{categorification}.  For each $n\geq m$ the $(A_k^{[n]},A_k^{[n-m]})$ bimodule $P^\lambda$ is flat, as is the bimodule $Q^\lambda$.  Summing over $n$, we have indued endomorphisms of the Grothendieck group
$$
	[P^{\lambda}], [Q^{\lambda}] : \cF_k \longrightarrow \cF_k,
$$

The proof of the following proposition is given in \cite{CL1}.  That reference was concerned with the particular case $k=2$, although the proof carries over with easy modification to general $k\geq1$.

\begin{proposition}\cite[Proposition 2]{CL1} \label{prop:Grothendieck} The operators $[P^{\lambda}], [Q^{\lambda}]$ satisfy the following properties.
\begin{enumerate}
\item $[P^{\lambda}]\circ[P^{\mu}] = \sum_\nu d_{\lambda\mu}^\nu [P^{\nu}]$, where $d_{\lambda\mu}^\nu\geq 0$ is the (ordinary) Littlewood-Richardson coefficient.
\item $[Q^{(n)}]\circ [P^{(m)}]
 = \sum_{l\geq 0} {{k+l-1}\choose{l}}_t [P^{(m-l)}]\circ [Q^{(n-l)}],$ where ${x\choose y}_t$ is the quantum binomial coefficient.
  \end{enumerate}
\end{proposition}
\noindent
In the above proposition, the quantum binomial coefficients are normalized to be symmetric about the origin, e.g.
${{4}\choose{2}}_t = t^{-4} + t^{-2} + 2 + t^2 + t^4$.

Now, considering symmetric functions instead of Grothendieck groups, we define endomorphisms 
$$
p^\lambda, q^\lambda : \Lambda_{t^k,t} \longrightarrow \Lambda_{t^k,t}
$$
by letting $p^\lambda$ be multiplication by the Schur function $s_\lambda$ and letting $q^\lambda$ be its adjoint with respect to $\langle,\rangle_{t^k,t}$.  

\begin{proposition}\label{prop:symmetric}
The operators $p^\lambda, q^\lambda$ satisfy the following properties.
\begin{enumerate}
\item $p^\lambda \circ p^\mu= \sum_\nu d_{\lambda\mu}^\nu p^\nu$, where $d_{\lambda\mu}^\nu\geq 0$ is the (ordinary) Littlewood-Richardson coefficient.
\item $q^{(n)}\circ p^{(m)} 
 = \sum_{l\geq 0} {{k+l-1}\choose{l}}_t p^{(m-l)}q^{(n-l)}$.
%where ${x\choose y}_t$ is the quantum binomial coefficient.
  \end{enumerate}
\end{proposition}

\begin{proof}
The first statement is clear.  
%For the second statement, 
%note that by Proposition \ref{prop:structure constants}, 
%if we expand the operator $p^{\lambda}q^{\mu}$ in the 
%basis $\{p^{\lambda}q^{\mu}\}_{\lambda,\mu}$, then the 
%coefficient  of $p^{\lambda}q^{\mu}$ is the generalized 
%Littlewood-Richardson coefficient $c_{\mu(n)}^{\lambda(m)}$.  
The second follows from example \ref{strip} and \eqref{hermsp}.
%Thus the claim follows immediately from Proposition 
%\ref{prop:special case}.
\end{proof}

Now Propositions \ref{prop:Grothendieck} and \ref{prop:symmetric} imply the theorem.  For, the inner product $\langle [H^\lambda] , [H^\mu] \rangle_{\cF_k}  $ can be computed as
$$
	 \langle [P^{(\lambda_1)}] [P^{(\lambda_2)}]\dots [P^{(\lambda_r)}] 1_0,
	[P^{(\mu_1)}] [P^{(\mu_2)}] \dots[P^{(\mu_s)}] 1_0 \rangle_{\cF_k} = $$
$$
	\langle [Q^{(\mu_1)}][P^{(\lambda_1)}] [P^{(\lambda_2)}]\dots [(P^{(\lambda_r)}] 1_0,
	[P^{(\mu_2)}] \dots[P^{(\mu_s)}] 1_0 \rangle_{\cF_k}
$$
where the last equality used the adjointness of $P^{(\mu_1)}$ and $Q^{(\mu_1)}$.
Now we use the second part of Proposition \ref{prop:Grothendieck} to write $Q^{(\mu_1)}[P^{(\lambda_1)}]  [P^{(\mu_2)}] \dots[P^{(\mu_s)}] 1_0$
as a sum of $[H^\kappa]$s for smaller $\kappa$, inductively determining the $\langle [H^\lambda] , [H^\mu] \rangle_{\cF_k}  $ in terms of inner products involving smaller partitions.

Similarly, the inner product 
$$
	\langle h_\lambda ,h_\mu\rangle_{k} = \langle p^{\lambda_1} p^{\lambda_2}\dots p^{\lambda_r}1_0,p^{\mu_1}p^{\mu_2}\dots p^{\mu_s}1_0 \rangle_{k}
$$
can be computed using the adjointness of $p^{(\mu_1)}$ and $q^{(\mu_1)}$, together with the second part of Proposition \ref{prop:symmetric}.
Since the structure constants in Propositions \ref{prop:Grothendieck} and \ref{prop:symmetric} agree, we conclude by induction that
$$
	\langle [H_\lambda],[H_\mu] \rangle_{\cF_k} = \langle h_\lambda, h_\mu \rangle_{k} = 
	\langle \Phi([H_\lambda]),\Phi([H_\mu]) \rangle_{k},
$$
as desired.


\begin{thebibliography}{99}
\bibitem{CL1}
S.~Cautis and A.~Licata,
\emph{Heisenberg categorification and Hilbert schemes}
 Duke Math. J. Volume 161, Number 13 (2012), 2469-2547.

\bibitem{CL2}
S.~Cautis and A.~Licata,
\emph{Vertex operators and 2-representations of quantum affine algebras}
 arXiv:1112.6189

\bibitem{F}
W.~Fulton
\emph{Young Tableaux}
Cambridge University Press, 1997

\bibitem{G}
L.~Geissinger. 
\emph{Hopf algebras of symmetric functions and class functions.}
 In Combinatoire et representation du groupe symetrique (Actes Table Ronde C.N.R.S., Univ. Louis-Pasteur Strasbourg,
Strasbourg, 1976), pages 168Ð181. Lecture Notes in Math., Vol. 579. Springer, Berlin, 1977.

\bibitem{Mac}
I.~Macdonald, 
\emph{Symmetric functions and Hall polynomials}, 
The Clarendon Press, Oxford University Press, New York, 1995.

\bibitem{SV}
P.~Shan and E.~Vasserot, 
\emph{Heisenberg algebras and rational double affine Hecke algebras}
J. Amer. Math. Soc. 25 (2012), no. 4, 959Ð1031.

\bibitem{Z}
A.V.~Zelevinsky, 
\emph{A generalization of the Littlewood-Richardson rule and the Robinson-Schensted-Knuth correspondence}, 
Journal of Algebra 69 (1), 82¡V94, 1981.
\end{thebibliography}
\end{document}